\documentclass[10pt]{amsart}
\usepackage{latexsym}
\usepackage{amsfonts}

\usepackage{amsmath,amsthm,amssymb}

\title[Stabilizers of $\R$-trees with free isometric actions of $\FN$]
{Stabilizers of $\R$-trees with free isometric actions of $\FN$}

\author[I.~Kapovich]{Ilya Kapovich}

\address{\tt Department of Mathematics, University of Illinois at
  Urbana-Champaign, 1409 West Green Street, Urbana, IL 61801, USA
  \newline http://www.math.uiuc.edu/\~{}kapovich/} \email{\tt
  kapovich@math.uiuc.edu}

\author[M.~Lustig]{Martin Lustig}\address{\tt Math\'ematiques
  (LATP), Universit\'e Paul C\'ezanne -Aix Marseille III,\\  av. Escadrille
  Normandie-Ni\'emen, 13397 Marseille 20, France} \email{\tt Martin.Lustig@univ-cezanne.fr}

\newtheorem{thm}{Theorem}[section] \newtheorem{lem}[thm]{Lemma}
\newtheorem{cor}[thm]{Corollary} 
\newtheorem{prop}[thm]{Proposition} \theoremstyle{definition}
\newtheorem{defn}[thm]{Definition}

\newtheorem{conv}[thm]{Convention} \newtheorem{rem}[thm]{Remark}

\def\epsilon{\varepsilon}
\def\phi{\varphi}

\newcommand{\Out}{\mbox{Out}}
\newcommand{\Aut}{\mbox{Aut}}
\newcommand{\inv}{^{-1}}

\newcommand{\Stab}{\mbox{Stab}}
\newcommand{\Id}{\mbox{Id}}

\newcommand{\FN}{F_N}   % F ou F_n ou F_N ?
\newcommand{\cvn}{\mbox{cv}_N}
\newcommand{\cvnbar}{\overline{\mbox{cv}}_N}
\newcommand{\CVN}{\mbox{CV}_N}
\newcommand{\CVNbar}{\overline{\mbox{CV}}_N}
\newcommand{\R}{\mathbb R}

\def\strutdepth{\dp\strutbox}
\def \ss{\strut\vadjust{\kern-\strutdepth \sss}}
\def \sss{\vtop to \strutdepth{
\baselineskip\strutdepth\vss\llap{$\diamondsuit\;\;$}\null}}

\def\strutdepth{\dp\strutbox}
\def \sst{\strut\vadjust{\kern-\strutdepth \ssss}}
\def \ssss{\vtop to \strutdepth{
\baselineskip\strutdepth\vss\llap{$\spadesuit\;\;$}\null}}

\def\strutdepth{\dp\strutbox}
\def \ssh{\strut\vadjust{\kern-\strutdepth \sssh}}
\def \sssh{\vtop to \strutdepth{
\baselineskip\strutdepth\vss\llap{$\heartsuit\;\;$}\null}}

\def\qed{\hfill\rlap{$\sqcup$}$\sqcap$\par}
\def\bar{\overline}

\def\strutdepth{\dp\strutbox}
\def \ss{\strut\vadjust{\kern-\strutdepth \sss}}
\def \sss{\vtop to \strutdepth{
\baselineskip\strutdepth\vss\llap{$\diamondsuit\;\;$}\null}}

\def\strutdepth{\dp\strutbox}
\def \sst{\strut\vadjust{\kern-\strutdepth \ssss}}
\def \ssss{\vtop to \strutdepth{
\baselineskip\strutdepth\vss\llap{$\spadesuit\;\;$}\null}}
\def\qed{\hfill\rlap{$\sqcup$}$\sqcap$\par}

\begin{document}

\begin{abstract}
We prove that if $T$ is an $\mathbb R$-tree with a minimal free
isometric action of $F_N$, then the
$\Out(F_N)$-stabilizer of the projective class $[T]$ is virtually cyclic.

For the special case where $T=T_+(\phi)$ is the forward limit tree of
an atoroidal iwip element $\phi\in \Out(F_N)$ this is a consequence of
the results of Bestvina, Feighn and Handel~\cite{BFH97}, via very
different methods.

We also derive a new proof of the Tits alternative for subgroups of $\Out(F_N)$ containing an iwip (not necessarily atoroidal): we prove that every such subgroup $G\le \Out(F_N)$ is either virtually cyclic or contains a free subgroup of rank two. The general case of the Tits alternative for subgroups of $\Out(F_N)$ is due to Bestvina, Feighn and Handel.
\end{abstract}

\thanks{The first author was supported by the NSF
  grant DMS-0904200}


\maketitle

%\tableofcontents

\section{Introduction}

The action of the mapping class group of a (closed) surface
on its Teichm\"uller space has been a central theme in geometry, topology and ergodic theory, and it has served as a model case for many related subjects.  One of those is the outer automorphism group $\Out(\FN)$ of a free group $\FN$ of finite rank $N \geq 2$, and its action on {\em Outer space} $\CVN$: This is the projectivized space of metric simplicial trees $T$, equipped with an action of $\FN$ by isometries that is free and minimal. It is compactified (just as is Teichm\"uller space) by adding a {\em Thurston boundary} $\partial \CVN$, and the points of this compactification $\CVNbar = \CVN \cup \partial \CVN$ are homothety classes $[T]$ of $\R$-trees $T$ provided with isometric $\FN$-actions that are minimal and very small. These terms are defined and discussed below in section \ref{sect:prelim}.

A boundary point $[T] \in \partial \CVN$ may well be given by an
$\R$-tree $T$ where the $\FN$-action is free; however, contrary to
trees in the ``interior'' $\CVN$, such a free action will not be
discrete.

Our main result is:

\begin{thm}\label{main}
Let $N\ge 2$, let $T$ be an $\mathbb R$-tree with a minimal free
isometric action of $F_N$,
and let
$[T] \in \CVNbar$
be the projective class of $T$. Then the stabilizer $\Stab_{Out(F_N)}([T])$ is virtually cyclic.

If, in addition, the $F_N$-orbits of branch points are not dense in $T$,
then $\Stab_{Out(F_N)}([T])$ is finite.
\end{thm}

Theorem~\ref{main} is established in Theorem~\ref{thm:vc} below, which
actually provides a more detailed description of the stabilizer
$\Stab_{Out(F_N)}([T])$.  We show in Theorem~\ref{thm:vc} that
in $\Stab_{Out(F_N)}([T])$ there is always a canonically defined finite
normal subgroup $P_T\triangleleft \Stab_{Out(F_N)}([T])$ such that either $\Stab_{Out(F_N)}([T])$ is finite
and $\Stab_{Out(F_N)}([T])=P_T$, or else $\Stab_{Out(F_N)}([T])$ is a
semi-direct product $\Stab_{Out(F_N)}([T])=P_T\rtimes \mathbb Z$. We
also show that, if the tree $T$ from Theorem~\ref{main} does not have
dense $F_N$-orbits of branch points, then $\Stab_{Out(F_N)}([T])$ is
finite and thus equal to $P_T$. A result of Wang and Zimmermann~\cite{WZ} shows that every finite subgroup of $\Out(F_N)$ has order $\le N! \cdot 2^N$. Therefore Theorem~\ref{thm:vc} implies that, for $T\in \cvnbar$ as in Theorem~\ref{main}, the stabilizer $\Stab_{Out(F_N)}([T])$ has a cyclic subgroup (either trivial or infinite cyclic) of index at most $N! \cdot 2^N$.

The study in~\cite{BFH97} of $\Out(F_N)$-stabilizers of
particular points in compactified Outer space (corresponding to forward limiting $\R$-trees, explained below)
was a starting point in the proof by Bestvina, Feighn and Handel of the Tits
Alternative for $\Out(F_N)$~\cite{BFH00,BFH05}. Their paper \cite{BFH97} was a substantial source of
motivation for the present paper.

Theorem~\ref{main} can be viewed as part of a general theme, which
originates from Kleinian groups and from Teichm\"uller theory, that aims to
investigate the dynamics of the action of elements or of subgroups of
$\Out(\FN)$ on the space $\CVNbar$ (or on related spaces), and then to
deduce algebraic information from the geometric data obtained.
In this spirit, a useful dynamic information about a group acting
on a (compact) space is the fact that the action of certain group
elements has {\em North-South dynamics}: By this we mean that there
are precisely two fixed points (the two ``poles''), and that every other
point has the ``north pole'' as forward limit point and the ``south
pole'' as backwards limit point. Even stronger implications are
possible if the convergence is uniform on compact subsets (for a
precise definition see Proposition \ref{prop:LL+} below). For example,
it is known that a pseudo-Anosov mapping class has such a strong form of North-South dynamics on the Thurston compactification of the
Teichm\"uller space of a hyperbolic surface.

For $\Out(\FN)$ the analogous elements to pseudo-Anosov mapping
classes are {\em atoroidal iwip automorphisms}, see Definition
\ref{defn:iwips}. In \cite{LL} it has been shown that their induced
dynamics on compactified Outer space $\CVNbar$ is precisely of this
North-South type with uniform convergence on compact subsets.

The two ``poles'', i.e. the fixed points $[T_+(\phi)], [T_-(\phi)] \in
\CVNbar$, of such an atoroidal iwip automorphism $\phi \in \Out(\FN)$
are given by $\R$-trees $T_+(\phi)$ and $T_-(\phi)$ for which the
isometric action of $\FN$ is free, so that our Theorem \ref{main}
applies. There is also a fairly explicit way of
understanding the forward limiting tree $T_+(\phi)$ of $\phi$ via train-track representatives of $\phi$
(see \cite{BFH97, BH92,GJLL}). As noted above, in \cite{BFH97}
Bestvina, Feighn and Handel proved that if $\phi\in \Out(F_N)$ is an
iwip (which they do not require to be atoroidal) then
$Stab_{Out(F_N)}([T_+(\phi)])$ is virtually cyclic.

Theorem~\ref{main} generalizes this result for atoroidal iwips. As noted below, we also recover the conclusion that
$Stab_{Out(F_N)}([T_+(\phi)])$ is virtually cyclic for iwips that are
not atoroidal, via a direct reduction of that case to known facts in
surface theory.

Note that Theorem~\ref{main} applies to a greater class of
trees than the forward limit trees of atoroidal iwips. In particular,
it is possible for a non-iwip to fix the projective class of a free
$F_N$-tree. For example, if $\phi\in Out(F_N)$ is an atoroidal iwip,
then the ``double'' $\bar\phi\in \Out(F_{2N})$ of $\phi$ is not an iwip
but the forward limit tree of $\bar\phi$ (obtain from doubling a
train-track representative of $\phi$ and then applying the same
construction as in \cite{GJLL}) has a free $F_{2N}$-action and is
projectively fixed by $\bar\phi$.

In \cite{BFH97} Bestvina, Feighn and Handel also introduced the notion
of a ``stable lamination'' $\Lambda_\phi^+$ of an iwip $\phi$, defined
explicitly in terms of a train-track representative of $\phi$. The
main technical result (Theorem 2.14) of \cite{BFH97} states that for an iwip $\phi\in
\Out(F_N)$ the $\Out(F_N)$-stabilizer of $\Lambda_\phi^+$ is virtually
cyclic. In \cite{KL4} we use Theorem~\ref{main} to recover this
result for atoroidal iwips via geodesic currents on $F_N$ and the
intersection form between $\R$-trees and currents (see
\cite{KL1,KL2,KL3,KL4,Martin,Ka1,Ka2} for background information
regarding geodesic currents on free groups).

In the case where $\phi\in \Out(F_N)$ is an iwip which is not
atoroidal, the action of $F_N$ on $T_+(\phi)$ is not free. However, it
is known~\cite{BH92} that such $\phi$ must come from a pseudo-Anosov
homeomorphism of a compact surface with a single boundary component.
It turns out that in this case one can reduce the proof that
$Stab_{Out(F_N)}([T_+(\phi)])$ is virtually cyclic to known facts
about mapping class groups (see Proposition~\ref{prop:toroidal}
below). Thus, Theorem~\ref{main} and Proposition~\ref{prop:toroidal}
imply the following result originally established in \cite{BFH97}\footnote{As far as we were able to understand, there seems to be a gap in the proof of the main technical result, Theorem 2.14, in \cite{BFH97}.  Namely, the arguments presented there seem insufficient for proving Proposition~2.6~(1) in \cite{BFH97} for the case where, for example, $\psi\in Stab(\Lambda)$ is a reducible polynomially growing automorphism. Specifically, the statement ``Notice that all $H_0$-segments are Nielsen (periodic) or else $h$-iteration will produce arbitrarily long leaf segments
contained in $H_0$ contradicting quasiperiodicity" in the proof Proposition~2.6~(1) in \cite{BFH97} is incorrect and a more involved substitute argument is required to complete the proof of Proposition~2.6~(1).  The gap is fillable and a subsequent paper~\cite{BFH00} of Bestvina, Feighn and Handel gives an independent and more detailed proof of generalizations of the main results from \cite{BFH97}, via more elaborate
train track arguments. Also, Arnaud Hilion suggested to us a different direct way of patching the proof of Proposition~2.6~(1) in \cite{BFH97}, via the improved relative train track methods developed in \cite{BFH00}. Our paper presents an alternative argument for stabilizers of forward limiting trees of iwips, avoiding the train track machinery altogether.}:

\begin{cor}\label{cor:vc}\cite{BFH97}
Let $N\ge 2$ and let $\phi\in \Out(F_N)$ be an iwip.

Then
$Stab_{Out(F_N)}([T_+(\phi)])$ is virtually cyclic.
\end{cor}

\medskip

The proofs in \cite{BFH97} rely on exploiting the train-track
machinery for elements of $\Out(F_N)$. Our proof of Theorem~\ref{main}
uses an alternative approach and uses the machinery of ``homotheties''
and ``eigenrays'' for trees projectively fixed by some elements of
$\Out(F_N)$. We will now give a brief overview of our argument:

If $\phi\in \Stab_{Out(F_N)}([T])$ then $T$ is $F_N$-equivariantly
isometric to the tree $\lambda(\phi) T\Phi$ where $\lambda(\phi)>0$ is
the ``stretching factor'' of $\phi$ and where $\Phi\in \Aut(F_N)$ is a
lift of $\phi$ to $\Aut(F_N)$. This means that there exists a
homothety $H:T\to T$ with stretching factor $\lambda(\phi)$
such that for every $g\in F_N$ and $x\in T$ we have
\[
H(gx)=\Phi(g)H(x).
\]
Such homotheties $H$
{\em represent} elements of the stabilizer
$\Stab_{Out(F_N)}([T])$, and they turn out to have a number of nice
properties (compare \cite{GJLL, LL,Lu2}), which are recalled below in section \ref{sect:eigenrays}.
In particular, if $H$ fixes a branch point of $T$ and a
``direction'' $d$ at that branch point, then $H$ possesses a
well-defined ``eigenray'' $\rho$ starting at $x$ in direction $d$ such
that $H(\rho)=\rho$, so that $H$ acts on the ray $\rho$ as multiplication by
$\lambda(\phi)$ does on $\mathbb R_{\ge 0}$. The stretching-factor map
$\lambda: \Stab_{Out(F_N)}([T])\to\ \mathbb R_{>0}$
is a group homomorphism. To prove Theorem~\ref{main} we show that the image
of $\lambda$ is cyclic and the kernel of $\lambda$ is finite. The finite normal subgroup
$P_T\triangleleft \Stab_{Out(F_N)}([T])$
mentioned above is precisely the kernel of the homomorphism $\lambda$:
\[
P_T=Ker(\lambda)=\{\phi\in \Stab_{Out(F_N)}([T]): \lambda(\phi)=1\}.
\]
Thus $P_T$ consists precisely of all those elements of $\Stab_{Out(F_N)}([T])$ which are represented by isometries of $T$.

Our proof that the image of $\lambda$ is cyclic can be pushed through
to work for the case of arbitrary very small tree $T\in \cvnbar$.\footnote{Gilbert Levitt has shown us that the fact, that the image of the map $\lambda$ is cyclic, can alternatively be derived from Theorem 4.3 of \cite{Lu1}.}
However, the argument that $Ker(\lambda)$ is finite relies crucially
on the fact that $T$ is a free $F_N$-tree.

In particular, the conclusion of Theorem~\ref{main} fails if we allow
very small trees with non-trivial point stabilizers and trivial arc
stabilizers. For example, if $T$ is the Bass-Serre tree corresponding
to a proper free product decomposition $F_N=A\ast B$ (where both $A$
and $B$ are nontrivial and at least one of $A$ or $B$ is non-cyclic),
then there are many automorphisms of $F_N$ that preserve this free
product decomposition and hence fix $T$ (and thus $[T]$).
Nevertheless, the proof of Theorem~\ref{main} provides an approach for
understanding stabilizers $\Stab_{Out(F_N)}([T])$ for arbitrary $T\in
\cvnbar$. Since the image of $\lambda$ is also cyclic here, the main
task becomes to understand the structure of the kernel $Ker(\lambda)$.

As a consequence of Theorem~\ref{main} and of the ``North-South"
dynamics of atoroidal iwips on $\CVNbar$~\cite{LL} we derive (see \S
\ref{sect:TA}) without much effort a new proof of the Tits Alternative
for subgroups of $\Out(F_N)$ that contain an arbitrary iwip (not
necessarily atoroidal):

\begin{cor}\label{TA}
Let $G\le \Out(F_N)$ be a subgroup such that $G$ contains
an iwip element $\phi$. Then exactly one of the following occurs:
\begin{enumerate}
\item The group $G$ is virtually cyclic.
\item There exists $g\in G$ and $M\ge 1$ such that $\langle \phi^M, g^{-1}\phi^M g\rangle \le G$ is free of rank two with free basis $\phi^M, g^{-1}\phi^M g$
\end{enumerate}
\end{cor}

This result has been proved in \cite{BFH97}.  The general
case of the Tits alternative for subgroups of $\Out(F_N)$ has been
established by Bestvina, Feighn and Handel in a series of deep papers
\cite{BFH00,BFH05} using the improved relative train-track technology.

\smallskip
\noindent
{\em Acknowledgements:}

The authors are grateful to Arnaud Hilion, Chris Leininger
and Gilbert Levitt for useful conversations.

\section{Preliminaries}\label{sect:prelim}

An $\R$-tree $T$ is a path-connected non-empty metric space, such that
for any two points $x, y \in T$ there is a unique embedded arc $[x, y]
\subseteq T$ which joins $x$ to $y$, and this arc is isometric to the
interval $[0, d(x, y)] \subseteq \R$. All $\R$-trees in this paper are
equipped with a (left) isometric action of the free group $\FN$ of finite rank $N \geq 2$. Such an $\R$-tree $T$ is called {\em minimal} if there is no non-empty $\FN$-invariant subtree $T' \subseteq T$ different from $T$.

For any element $w \in \FN$ the {\em translation length} on $T$ is defined by
\[|| w ||_T = \inf_{x \in T} \{ d(x,wx)\}\ .\]
This infimum is always attained, and in the case where $||w||_T > 0$ the
set of points $x \in T$ which realize $d(x,wx) = || w ||_T$ is
isometric to $\R$ and is called the {\em axis} of $w$, denoted by
${\rm Ax}(w)$.  An element $w \in \FN$ (or more precisely: the isometry $T
\to T, x \mapsto w x$) is called {\em hyperbolic} if $||w||_T>0$ and
\emph{elliptic} if $||w||_T=0$.

An $\R$-tree $T$ with an isometric $F_N$-action is called {\em small} if for any $x \neq y$ in $T$ the stabilizer $\Stab_{\FN}([x,y]) \subseteq \FN$ is cyclic. The tree $T$ is {\em very small} if in addition no $w \in \FN$ inverts a non-degenerate segment or fixes a non-degenerate tripod in $T$.

The following is well known (see, for example, \cite{Pau}).

\begin{lem}\label{axis}
Let $T$ be an $\R$-tree equipped with a minimal nontrivial
(i.e. without a global fixed point) isometric action of $F_N$. Then
$T$ is equal to the union of the axes ${\rm Ax}(w)$ for all hyperbolic $w \in \FN$.
\qed
\end{lem}

The \emph{unprojectivized Outer space} $\cvn$ consists of all $\R$-trees equipped with a minimal free discrete isometric actions of $F_N$.
Two such trees are considered as equal if there exists an $F_N$-equivariant isometry between them. The closure $\cvnbar$ of $\cvn$ in the equivariant Gromov-Hausdorff convergence topology is known~\cite{CL,BF93} to consists of all very small minimal isometric actions of $F_N$ on $\mathbb R$-trees, where again two trees are considered to be equal if there exists an $F_N$-equivariant isometry between them. Although the trivial action of $F_N$ on a tree consisting of a single point can be realized as the limit of free and discrete $F_N$-trees, this action by convention is excluded from $\cvnbar$, so that all points of $\cvnbar$ represent non-trivial minimal actions of $F_N$.

There is a natural (right) action of $\Aut(\FN)$ on $\cvnbar$ that
leaves $\cvn$ invariant. Namely, for $\Phi\in \Aut(F_N)$ and $T\in
\cvnbar$, the point $T\Phi\in \cvnbar$ is defined as follows.
The tree  $T\Phi$ is equal to $T$ as a metric space, but the action of
$F_N$ is twisted via $\Phi$:
\[
w \underset{T\Phi}{\cdot} x:= \Phi(w) \underset{T}{\cdot} x \text{  for any } w\in \FN, x\in T.
\]

It is easy to see that $Inn(\FN)$ is contained in the kernel of this action of $\Aut(F_N)$ on $\cvnbar$ and therefore the action factors through to the action of $\Out(F_N)$ on $\cvnbar$ as follows: for $T\in \cvnbar$ and $\phi\in \Out(F_N)$ we have $T\phi:=T\Phi$ where $\Phi\in \Aut(F_N)$ is any representative of $\phi$.

The \emph{projectivization} $\CVNbar$ of $\cvnbar$ is defined as $\CVNbar=\cvnbar/\sim$, where for $T_1, T_2\in \cvnbar$ we have $T_1\sim T_2$ if there exists a constant $c>0$ such that $T_1=cT_2$. The latter condition means that there exists an $F_N$-equivariant isometry between $T_1$ and the tree $cT_2$, which is obtained from $T_2$ by multiplying the metric on $T_2$ by $c$, while using the same $F_N$-action as given on $T_2$. The $\sim$-equivalence class of $T\in \cvnbar$ is denoted by $[T]$.  The image of $\cvn$ in $\CVNbar$ under the canonical projection map is denoted by $\CVN$. Thus $\CVN$ is the projectivization of $\cvn$ and $\CVN=\{[T] \mid T\in \cvn\}$.

The actions of $\Aut(F_N)$ and $\Out(F_N)$ respect the $\sim$-equivalence relation and hence they quotient through to actions on $\CVNbar$: for $\phi\in \Out(F_N)$ and $[T]\in \CVNbar$ we have $[T]\phi:=[T\phi]$.

For $T\in \cvnbar$ and $x\in T$ we define the \emph{valence} $val(x)$  of $x$ to be the number of
connected components of $T \smallsetminus\{x\}$. These connected components
themselves are called \emph{directions} at $x$.
We can also think of a direction at $x$ as an equivalence
class of nondegenerate geodesic segments starting at $x$, where two
such segments are equivalent if they have an overlap of positive length.

The following is well known and follows directly from the definition of an $\R$-tree:

\begin{lem}
\label{common-segments}
Let $T\in \cvnbar$.
Then for any two points $y, y'$ contained in the same direction $d$ at
some point $x \in T$, the segments $[x, y]$ and $[x, y']$ intersect in
a non-degenerate segment
\[
[x, y] \cap [x, y'] = [x, z] \qquad {\rm with} \qquad  x \neq z.
\]
\qed
\end{lem}

 Note that for every $x\in T$ we have $val(x)\ge 2$. Indeed, if $val(x)=1$, then we can remove the $F_N$-orbit of $x$ from $T$ to get a proper $F_N$-invariant subtree, contradicting the minimality of the action of $F_N$ on $T$.
We say that $x\in T$ is a \emph{branch point} if $val(x)\ge 3$.

Note that for $T\in \cvnbar$, the group $F_N$ acts on the set of branch points and on the set of directions at branch points in $T$.
We will need the following useful fact:

\begin{thm} [Gaboriau-Levitt \cite{GL}]
\label{GL}
Let $T \in \cvnbar$ be arbitrary. Then there are finitely many $\FN$-orbits of branch points, and
only finitely many $Stab_{\FN}(Q)$-orbits of directions at any branch point $Q$. In particular, there are only finitely many $\FN$-orbits of directions at branch points.
\end{thm}

The result of Gaborau-Levitt is actually much more specific, in that it gives an upper bound formula in terms of what they call the {\em index} of $T$. For the case where the $\FN$-action on $T$ is free, this formula reduces to the following:

\begin{cor} [Gaboriau-Levitt \cite{GL}]
\label{GL1}
Let $T \in \cvnbar$ be with free $\FN$-action. Then the following holds:
\begin{enumerate}
\item
The number of $\FN$-orbits of branch points is bounded above by $2N-2$.
\item
For each branch point $Q \in T$, the number of directions at $Q$ is bounded above by $2N$.
\item
The total number of $\FN$-orbits of directions at branch points in $T$ is bounded above by $6N-6$.
\end{enumerate}
\end{cor}

\section{Stretching factors, homotheties and eigenrays}
\label{sect:eigenrays}

In this section we present some of the basics of $\R$-trees with
isometric action of a free group $\FN$ of finite rank $N \geq 2$. The
material of this section is known to the experts, but it is a little
scattered in the literature (see e.g. \cite{GJLL, Lu2}).

%\subsection{Homotheties}

Recall here that a {\em homothety with stretching factor $\lambda >
  0$} is a bijection between $H:T\to T'$ metric spaces $T$ and $T'$
which satisfies $d(Hx,Hy)=\lambda d(x,y)$ for any $x,y\in T$.

\smallskip
\begin{defn}[Stretching factors and homotheties]
\label{s-f-a-h}
Let $T\in \cvnbar$ and let $\Phi\in \Aut(F_N)$ be such that $[T]\Phi=[T]$, or equivalently,
$\Phi\in \Stab_{Aut(\FN)}([T])$. Thus for some
$\lambda=\lambda(\Phi)>0$, called the \emph{stretching factor of
  $\Phi$}, there exists an $F_N$-equivariant isometry between the
trees $\lambda T$ and $T\Phi$.

By definition of $T\Phi$ this means that there is
 a homothety $H=H_\Phi: T\to T$ with stretching factor $\lambda$,
such that
\[
H(wx)=\Phi(w) H(x) \text{ for any } x\in T, w\in \FN \tag{$\dag$}.
\]

In this case we say that $H$ is a \emph{homothety of $T$
representing $\Phi$}.

If $\phi\in \Out(F_N)$ is such that $[T]\phi=[T]$, if $\Phi\in
\Aut(F_N)$ is a lift of $\phi$ to $\Aut(F_N)$ and if $H:T\to T$ is a
homothety of $T$  representing $\Phi$, we will also say that $H$ is a \emph{homothety of $T$
representing $\phi$}. We also put $\lambda(\phi)=\lambda(\Phi)$ in
this case and call $\lambda(\phi)$ the \emph{stretching factor of $\phi$}.
\end{defn}

If $T\Phi=\lambda_1T=\lambda_2T$ then $\lambda_1=\lambda_2$. The
reason is that in this case, at the level of translation length
functions, one has $||\/ \cdot
\/||_{T\Phi}=\lambda_1||\cdot||_T=\lambda_2|| \cdot ||_T$ on
$F_N$. Thus for $\Phi\in \Stab_{Aut(\FN)}([T])$ the stretching factor
$\lambda(\Phi)$ is well defined. It is also easy to see that for
$\phi\in \Stab_{Out(F_N)}([T])$
the stretching factor $\lambda(\phi)$ is well-defined.

Note that a homothety of $T$ with stretching factor $\lambda = 1$ is an isometry of $T$.

The following is well known:
%~\cite{GJLL}.

\begin{lem}\cite{GJLL}
\label{dense-branch-points}
 Let $T \in \cvnbar$ and let $H: T \to T$ be a homothety with stretching factor $\lambda \ne 1$
which represents some automorphism of $\FN$. Then:
\begin{enumerate}
\item The branch points are dense in $T$ (in fact, for every branch point its $F_N$-orbit is dense in $T$).
\item The stabilizer in $\FN$ of any non-degenerate segment is trivial.
\end{enumerate}

\end{lem}

\begin{cor}\label{cor:dense}
Let $T \in \cvnbar$ and $H: T \to T$ be as in
Lemma~\ref{dense-branch-points}. Then for any $x\ne y$ in $T$ the
branch points of $T$ are dense in $[x,y]$.
\end{cor}

\begin{proof}
Let $x\ne y$ be two points in $T$. Let $z\in [x,y]$ such that $z\ne
x$, $z\ne y$ be arbitrary. We claim that there exist branch points of
$T$ in $[x,y]$ that are arbitrary close to $z$.

Let $\epsilon>0$  be any such that
$10\epsilon<\min\{d(z,x),d(z,y)\}$. By Lemma~\ref{dense-branch-points}
there exists a branch point $q$ of $T$ such that $d(z,q)\le
\epsilon$. If $q\in [x,y]$, we are done. If $q\not\in [x,y]$, let
$q'\in [x,y]$ be the nearest point projection of $q$ to $[x,y]$. By
the choice of $\epsilon$ and of $q$ we see that $q'\ne x$, $q'\ne y$,
$d(q',z)\le \epsilon$ and that $q'$ is a branch point of $T$. This
establishes the claim and completes the proof.
\end{proof}

The following is essentially an immediate corollary of the definitions:

\begin{lem}\label{lem:basic}
For any $T \in \cvnbar$ we have:
\begin{enumerate}
\item Let $\Phi_1,\Phi_2\in \Stab_{Aut(\FN)}([T])$ and let $H_1,H_2$ be homotheties of $T$ representing $\Phi_1$ and $\Phi_2$ accordingly. Then $H_1H_2$ is a homothety representing $\Phi_1\Phi_2$ and thus $\lambda(\Phi_1\Phi_2)=\lambda(\Phi_1)\lambda(\Phi_2)$.

 \item Suppose $H$ is a homothety representing $\Phi\in \Stab_{Aut(\FN)}([T])$. Then for any $u\in F_N$ the homothety $uH$ represents
 $\Phi_1= I_u \circ \, \Phi\in \Stab_{Aut(\FN)}([T])$,  where
 $I_u \in Inn(\FN)$ is the inner automorphism defined by $I_u(w)=u w u^{-1}$ for every $w\in F_N$.

 \item Let $\Phi\in \Stab_{Aut(\FN)}([T])$ and let $H$ be a homothety representing $\Phi$. Then $H^{-1}$ is a homothety representing $\Phi^{-1}$.
 \end{enumerate}
 \qed
\end{lem}

Lemma~\ref{lem:basic} implies that
\[
\lambda: \Stab_{Aut(\FN)}([T]) \to \mathbb R_{>0}
\]
is a homomorphism
to the multiplicative group $\R_{>0}$,
and that for every $\Phi\in Inn(F_N)$ we have $\lambda(\Phi)=1$.

Thus considered as a function on $\Stab_{Out(\FN)}([T])$, the stretching
factor map
\[
\lambda: \Stab_{Out(\FN)}([T]) \to \mathbb R_{>0}
\]
is also a homomorphism.

It is easy to see that any homothety $H$ representing some $\Phi\in \Stab_{Aut(\FN)}([T])$ takes geodesic segments to geodesic segments, preserves valences of points of $T$, takes branch points to branch points and directions at branch points to directions at branch points. Moreover, equation $(\dag)$ of Definition \ref{s-f-a-h}
implies that $H$ acts by permutations on $F_N$-orbits of branch points and on $F_N$-orbits of directions at branch points.

\begin{lem}\label{fix}
Let $T\in \cvnbar$ and let $H$ be a homothety representing some $\Phi\in \Stab_{Aut(\FN)}([T])$ such that $H$ preserves every $F_N$-orbit of branch points and every $F_N$-orbit of directions at branch points.
Suppose $H$ fixes a branch point $x$ of $T$.

If the $\FN$-action on $T$ is free, then
$H$ fixes every direction at $x$.
\end{lem}
\begin{proof}
Suppose $d$ is a direction at $x$.
Since $H$ fixes the $\FN$-orbit of every direction in $T$, we have
$Hd=ud$ for some $u\in F_N$. Hence $ux=x$ and therefore $u=1$, since the action of $F_N$ on $T$ is free. Thus $Hd=d$, as required.
\end{proof}

\begin{lem}\label{key1}
Let $T \in \cvnbar$, and
let $H$ be a homothety with streching factor $1$ (that is
an isometry of $T$)  that represents
$\Phi\in Stab_{\Aut(F_N)}([T])$.
Then the following hold:

\begin{enumerate}
\item
If $\Phi = \Id_{\FN} \in \Aut(\FN)$, then $H=\Id_T$.
\item If $H$ represents $I_u \in Inn(F_N)$, then $H(x)=ux$ for every $x\in T$.

\end{enumerate}

\end{lem}

\begin{proof}
%$ $
%
%\noindent
(1) Let $w\in F_N$ be a hyperbolic element. Then, since
  $\Phi = \Id_{\FN} \in \Aut(\FN)$, by (\dag) of Definition \ref{s-f-a-h}
  we have $H {\rm Ax}(w)=Hw{\rm Ax}(w)=wH {\rm Ax}(w)$.  Thus $w$ preserves the line
  $H Ax(w)$, and therefore $H {\rm Ax}(w)={\rm Ax}(w)$ since ${\rm Ax}(w)$
  is the only $w$-invariant line in $T$.

  Since $H$ is an isometry of $T$, it is either hyperbolic or
  elliptic. If $H$ is a hyperbolic isometry, then it preserves a
  unique line in $T$, namely ${\rm Ax}(H)$, which contradicts the fact
  that $H$ leaves invariant the axis of every hyperbolic element of
  $F_N$. Thus $H$ is elliptic. Hence for every hyperbolic element $w$
  of $F_N$ the isometry $H$ either fixes pointwise or acts as a
  reflection on the axis of $w$. We claim that in fact $H$ fixes every
  axis of a hyperbolic element pointwise. If not, then there exists a
  hyperbolic element $w\in F_N$ such that $H$ acts as a reflection on
  the axis of $w$. Let $x_0\in {\rm Ax}(w)$ be the unique point of
  ${\rm Ax}(w)$ fixed by $H$. Since $H$ represents $\Id_{\FN}$, we have
  $H(wx_0)=wHx_0=wx_0$, so that $H$ fixes $wx_0$. However, $wx_0\in
  {\rm Ax}(w)$ and $d(wx_0,x_0)=||w||_T>0$, so that $wx_0\ne x_0$. This
  contradicts the fact that $H$ acts a reflection centered at $x_0$ on
  ${\rm Ax}(w)$. Thus the claim is established and $H$ fixes every axis
  pointwise. Then it follows from Lemma~\ref{axis} that $H = Id_T$.

\smallskip
\noindent (2) Suppose now that $H$ is an isometry of $T$ that
represents the inner automorphism $I_u \in \Aut(F_N)$.  Then, by
Lemma~\ref{lem:basic}, $u^{-1}H$ is an isometry that represents
$\Id_{F_N}\in \Aut(F_N)$. Therefore, by part (1), $u^{-1}H=\Id_T$, so
that $H(x)=ux$ for every $x\in T$, as required.
\end{proof}

\begin{cor}\label{cor:unique}
For any $T \in \cvnbar$ we have:
\begin{enumerate}
\item Let $\Phi \in \Stab_{Aut(\FN)}([T])$ and let $H_1,H_2$ be two homotheties of $T$ representing $\Phi$. Then $H_1=H_2$.
\item Let $\phi \in \Stab_{Out(\FN)}([T])$ and let $H_1,H_2$ be two homotheties of $T$ representing $\phi$. Then there is $u\in F_N$ such that $uH_1=H_2$.
\end{enumerate}
\end{cor}

\begin{proof}
(1) Both $H_1$ and $H_2$ are $F_N$-equivariant isometries between the trees $T$ and $\lambda(\Phi)\inv T\Phi$.

Hence, by Lemma \ref{lem:basic} (2), the map
$H:=H_2^{-1}H_1: T\to T$ is an $F_N$-equivariant isometry of $T$, that is $H(wx)=wH(x)$ for every $w\in F_N$ and every $x\in T$.

In particular $H$ is an isometry of $T$ representing the identity
$\Id_{F_N}\in \Aut(F_N)$. Part (1) of Lemma~\ref{key1} implies that $H=\Id_{T}$, so that $H_1=H_2$, as required.

\smallskip
\noindent (2) There are representatives $\Phi_1,\Phi_2\in \Aut(F_N)$ of $\phi$ such that $H_1$ represents $\Phi_1$ and $H_2$ represents $\Phi_2$. Since $\Phi_1, \Phi_2$ both represent $\phi$, there is $u\in F_N$ such that $\Phi_2(g)=u\Phi_1(w)u^{-1}$ for every $w\in F_N$.

Therefore by Lemma~\ref{lem:basic} the isometry $uH_1$ represents $\Phi_2$. Hence by part~(1) we have $uH_1=H_2$, as required.
\end{proof}

\begin{rem}
\label{sumup}
Part (2) of Corollary~\ref{cor:unique} and Lemma~\ref{lem:basic} imply that if $\phi\in \Stab_{Out(\FN)}([T])$ and $H$ is a homothety representing $\phi$ then a homothety $H'$ represents $\phi$ if and only if $H'$ has the form $H'=uH$ where $u\in F_N$.
\end{rem}

\begin{conv}[Subgroup $K_T$]\label{conv:K_T}
Let $T\in \cvnbar$.
As we have already observed, any homothety $H$ representing some $\phi\in \Stab_{Out(\FN)}([T])$ acts by permutations on $F_N$-orbits of branch points and on $F_N$-orbits of directions at branch points. Let $\mathcal D$ be
the set of $F_N$-orbits of
directions at branch points of $T$.
Thus there is a natural homomorphism from $\Stab_{Out(\FN)}([T])$ to the group $Sym(\mathcal D)$ of permutations of $\mathcal D$.
We denote by $K_T$ the kernel of this homomorphism.
\end{conv}

By Theorem~\ref{GL} the set $\mathcal D$ is finite, with upper bound to its cardinality given by Corollary \ref{GL1} (c), so that we obtain:

\begin{cor}\label{cor:K}
Let $T\in \cvnbar$. Then the subgroup $K_T\subseteq \Stab_{Out(\FN)}([T])$
is of finite index.
\qed
\end{cor}

From now on we will restrict
most of our attention to automorphisms in $K_T$.

\begin{rem}\label{rem:good}
If $\phi\in K_T$, $x\in T$ is a branch point and $H'$ is a homothety representing $\phi$, we can always choose another homothety $H=wH'$ representing $\phi$, for some $w\in F_N$, such that $H$ fixes $x$. Namely, because $H'$ preserves the $F_N$-orbit of $x$, there is $u\in F_N$ such that $H'(x)=ux$. Then $H=u^{-1}H'$ fixes $x$, as required.
\end{rem}

%\subsection{Eigenrays}

\begin{defn}[Eigenray]

Let $T\in \cvnbar$ and let $H$ be a homothety representing some
automorphism $\phi\in \Stab_{Out(\FN)}([T])$ with $\lambda(\phi)>1$. A
(closed) geodesic ray $\rho \subseteq T$, which starts at some point $x
\in T$ such that $\rho-\{x\}$ is contained in a direction $d$ at $x$, is called an {\em eigenray of $H$ at $x$ in the direction $d$}, if one has:
\[
H(\rho) = \rho.
\]
In this case it follows that $H(x) = x$ and $H(d) = d$.
Furthermore, note that $H$ acts on $\rho$ as multiplication by $\lambda$ acts on $\mathbb R_{\ge 0}$.
\end{defn}

\begin{prop}
\label{eigenrays}
Let $T\in \cvnbar$ and let $H$ be a homothety representing some $\phi\in
\Stab_{Out(\FN)}([T])$ with $\lambda(\phi)>1$. Let $d$ be a direction
at $x\in T$, and assume that $H(x) = x$ and $H(d) = d$.

Then there exists a unique eigenray of $\rho=\rho_d$ of $H$ in $T$ which starts at $x$ in direction $d$.
\end{prop}

\begin{proof}
Since $H(d) = d$, it follows for any point $y \in d$ that the segments $[x, y]$ and $H[x, y] = [x, H(y)]$ overlap non-trivially, by Lemma \ref{common-segments}. Consider a point $z \neq x$ in $[x, y] \cap H[x, y]$. Then $[x, z]$ is a subsegment of $H[x, z] = [x, H(z)] \subseteq [x, H(y)] \cap [x, H^2(y)]$, and the infinite nested union $\cup\{ [x, H^n(z)] \mid n \in N\}$ forms a ray $\rho$ which by construction satisfies $\rho = H(\rho)$, i.e. it is an eigenray at $x$ in the direction $d$.

The uniqueness of $\rho$ follows from the fact that another such
eigenray $\rho'$ must have (by Lemma \ref{common-segments}) a
non-degenerate initial segment in common with $\rho$, but the
bifurcation point (that is, the endpoint of the maximal common initial segment) $y$ must have $H$-image contained in both, $\rho$ and $\rho'$ (by the $H$-invariance of either).  Since $d(x, H(y)) = \lambda(\phi) d(x, y) > d(x,y)$, this yields a contradiction to the above definition of the point $y$.
\end{proof}

The following proposition shows how eigenrays transform under the
action of elements of $\FN$ and under the action by other homotheties.
The statement of this proposition is known (see \cite{GJLL}) but we
present a proof for completeness:

\begin{prop}\label{well-known1}
Let $T\in\cvnbar$ and let $H$ be a homothety that represents some
$\Phi\in Aut(F_N)$ such
that the outer automorphism
class $\phi$
of $\Phi$ belongs to $K_T$, and
such that $\lambda=\lambda(\phi)>1$. Suppose that $H$ fixes a branch point $x\in T$. Let $\rho=\rho_d$ be the eigenrary of $H$ in a direction $d$ at $x$. Then the following hold:

\begin{enumerate}
\item Assume that the $\FN$-action on $T$ is free.
Suppose that  $H' = w H$, for some $w \in \FN$, is such that $H'$ fixes a branch point $x' = v x$, with $v \in \FN$.
Then $H'$ fixes the
direction $d' = v d$ at $x'$, and $v \rho_d$ is the eigenray of $H'$ in the direction $d'$.
Furthermore, in this case $w = v \, \Phi(v)\inv$.

\item Let $H'$ be a homothety of $T$ which represents some element $\psi$ of
$\Stab_{Out(F_n)}([T])$.
Then $H_1=H'H H'^{-1}$ is another homothety that
represents $\psi\phi\psi^{-1}\in K_T$, the stretching
factors of $H$ and $H_1$ are equal, $H_1$ fixes the point $H'(x)$ and the direction $H'(d)$ at $H'(x)$ and, moreover, $H'(\rho)$ is the eigenray of $H_1$ at $H'(x)$ in the direction $H'(d)$.

\item  Let $H'$ be another homothety of $T$ which represents
some element $\psi$ of $K_T$ with $\lambda'=\lambda(\psi)>1$
such that $H'$ fixes $x$ (and hence $H'$ fixes every direction at $x$). Let $\rho'=\rho'_d$ be the eigenray of $H'$ in direction $d$ and let $C>0$ be the length of the maximal common initial segment of $\rho_d$ and
$\rho'_{d}$.

Then the homothety
$H''=H H' H^{-1}$
fixes $x$ and represents $\phi\psi\phi^{-1}$. Let $\rho''=\rho''_d$ be the eigenray of $H''$ in direction $d$. Then $\rho$ and $\rho''$ have a common initial segment of length $\lambda(\phi) C$.
\end{enumerate}
\end{prop}

\begin{proof}
%$ $
%
\noindent (1) Since $H' = w H$, it follows that $H'$ also represents $\phi \in K_T$,
so that $H'$ fixes every $\FN$-orbit of directions of $T$. Thus, since
by assumption the action of $F_N$ on $T$ is free, Lemma~\ref{fix}
implies that $H'$ fixes any direction at its fixed point $x'$, so that
we have $H'(d') = d'$.

 We deduce:
\[
d' = H'(d') = H'(v d) = w H v(d) = w \Phi(v) H(d) = w \Phi(v) d = w \Phi(v) v\inv d'.
\]
Thus the isometry of $T$ given by the action of the element $w \Phi(v)
v\inv$ fixes the direction $d'$ and therefore it fixes the initial
point $x'=vx$ of $d'$.  Since the action of $F_N$ on $T$ is free, it
follows that $w \Phi(v) v\inv=1$ in $F_N$, so that $w = v \Phi(v)\inv$.

Then:
\[
H'(v \rho_d) = w H v (\rho_d) = w \Phi(v) H(\rho_d) = v H(\rho_d) = v \rho_d \, ,
\]
which shows that $v \rho_d$ is the eigenray of $H'$ at $x'$ in the direction of $d'$.

\smallskip
\noindent
(2)
We note that
\[
H_1(H'(\rho))=H'H H'^{-1} H' (\rho)=H' H (\rho) =H'(\rho).
\]
Thus $H_1(H'(\rho))=H'(\rho)$ which, since $H_1$ is a homothety and $H'(\rho)$ is a ray at $H'(x)$ in the direction $H'(d)$, implies the statement of part (2) of the proposition.

\smallskip
\noindent
(3)
This is a direct consequence of part (2). The length estimate for the common initial segment of  $\rho$ and $\rho''$ follows from that for $\rho$ and $\rho'$ and from the stretching property of the homothety $H$.
\end{proof}

\section{Proof that stabilizers of projectivized free $F_N$-trees are virtually cyclic}

Our strategy for the proof that $\Stab_{Out(F_N)}([T])$ is virtually
cyclic if $T\in \cvnbar$ is a free $F_N$-tree will be to show that the
map
\[
\lambda: \Stab_{Out(\FN)}([T]) \to \mathbb R_{>0}
\]
has image $Im(\lambda)\subseteq \mathbb
R_{>0}$ which is cyclic, and that the kernel $Ker(\lambda)\subseteq
\Stab_{Out(\FN)}([T])$ is finite. This will imply that
$\Stab_{Out(\FN)}([T])$ is virtually cyclic.

Recall from Convention \ref{conv:K_T}
%Corollary \ref{cor:K} and the preceding discussion
that
an outer automorphism $\phi\in \Out(\FN)$ belongs to the finite index normal subgroup $K_T \le \Stab_{Out(F_N)}([T])$
if and only if any homothety $H$ representing $\phi$ preserves every $F_N$-orbit of branch points in $T$, as well as every $F_N$-orbit of directions at a branch point in $T$.

\begin{prop}\label{key2}
Let $T \in \cvnbar$,
and suppose that the $\FN$-action on $T$ is free.
Let $H$ be an isometry of $T$
that represents some $\Phi\in K_T \subseteq \Stab_{Aut(\FN)}([T])$. Suppose that $H$
fixes a branch point $x_0$ of $T$. Then $H=\Id_T$ and $\Phi=\Id_{F_N}$.
\end{prop}

\begin{proof}

We define $T_0:=Fix(H)$ to be the set of fixed points of $H$. Thus $T_0\subseteq T$ is a closed subtree of $T$ which contains $x_0$. We claim that $T=T_0$.
Indeed, by way of contradiction let us suppose $T_0\ne T$. 

Then there exists a point $y\in T-T_0$. Let $z$ be the nearest point
projection of $y$ to $T_0=Fix(H)$. Thus $Hz=z$ and $[z,y]\cap [z,
Hy]=\{z\}$. Suppose first that $z\ne x_0$.  Then $z$ is a branch point
since the directions at $z$ determined by $[z,x_0]$, $[z, y]$ and $[z,Hy]$
are distinct. Now Lemma 3.5 implies that $H$ fixes every direction at $z$
and hence $H$ must fix a non-degenerate initial segment of $[z,y]$,
contrary to the fact that $[z,y]\cap [z, Hy]=\{z\}$. Thus $z=x_0$.
However, $x_0$ is a branch point and hence by Lemma 3.5 $H$ must fix a
non-degenerate initial segment of $[x_0,y]$, again yielding a
contradiction.

Thus indeed, $T=T_0$, so that $H=\Id_T$. Using the fact that $H$ represents $\Phi$, we get
\[
wx=H(wx)=\Phi(w)H(x)=\Phi(w)x \text{  for every  } x\in T, w\in F_N.
\]

Since the action of $F_N$ on $T$ is free, this implies that $w=\Phi(w)$ for every $w\in F_N$. Thus $\Phi=\Id_{F_N}$
as claimed.
\end{proof}

Recall from Lemma \ref{lem:basic} and the subsequent discussion that
\[
\lambda: \Stab_{Out(\FN)}([T]) \to \mathbb R_{>0}
\]
is the stretching factor homomorphism.

\begin{prop}\label{finite}
If the $\FN$-action on $T \in \cvnbar$ is free, then  \[Ker(\lambda|_{K_T})=\{
1_{Out(F_N)}\}.\]
\end{prop}

\begin{proof}
Suppose $\phi \in Ker(\lambda|_{K_T})$, that is $\phi\in K_T$ and $\lambda(\phi)=1$. This means that every homothety $H$ representing $\phi$ is actually an isometry of $T$.

Recall that, since $\phi\in K_T$, every homothety which represents $\phi$ acts as an identity permutation on the set of $F_N$-orbits of branch points and directions at branch points in $T$. By Remark~\ref{rem:good}, we can find a lift $\Phi$ of $\phi$ to $\Aut(F_N)$ and a homothety $H$ representing $\Phi$ such that $H$ fixes some branch point of $T$.

Thus we can apply
Proposition~\ref{key2} to obtain that $H=\Id_T$ and $\Phi=\Id_{F_N}$.

Therefore
$\phi$ is the trivial outer automorphism,
and $Ker(\lambda|_{K_T})=\{
1_{Out(F_N)}\}$ as required.
\end{proof}

\begin{prop}
\label{discrete-image}
Let $T\in \cvnbar$,
and assume that the $\FN$-action on $T$ is free.

Then the set $\lambda(K_T)\subseteq \mathbb R_{>0}$ is a cyclic subgroup of
the multiplicative group
$\mathbb R_{>0}$.
\end{prop}

\begin{proof}
Suppose, on the contrary, that the subgroup $\lambda(K_T)\subseteq \mathbb R_{>0}$ is not cyclic. Since $(\mathbb R_{>0},\cdot)$ is isomorphic to $(\mathbb R, +)$, it follows that a subgroup of $(\mathbb R_{>0},\cdot)$ is either discrete and cyclic or else it is dense in $\mathbb R_{>0}$. Thus the subgroup $\lambda(K_T)\subseteq \mathbb R_{>0}$ is dense.

Therefore we can find a sequence $\psi_i\in K_T$ such that $\lambda(\psi_i)\in [2, 2.001]$ are all distinct and $\lim_{i\to\infty} \lambda(\psi_i)=2$. Since there are only finitely many $F_N$-orbits of branch points and directions at branch points, by part (3) of Proposition~\ref{well-known1} we can find $\phi_1=\psi_1$ and $\phi_i=\psi_1^{n_i} \psi_i \psi_1^{-n_i}$ for $i\ge 2$, with the following property:

Whenever $x$ is a branch point of $T$, $d$ is a direction at $x$,
$\widehat H_i$ are homotheties fixing $x$ and representing $\phi_i$
and $\widehat\rho_i$ are their eigenrays in the direction $d$ then for
every $i\ge 2$ the rays $\widehat\rho_i$ and $\widehat\rho_1$ have an overlap of length at least $100$.

Note that $\lambda(\phi_i)=\lambda(\psi_i)$ for every $i\ge 1$.

Now choose a branch point $Q$ of $T$ and a direction $d$ at $Q$.
For every index $i \geq 1$ let
$H_i$ be the homothety representing $\phi_i$ and fixing $Q$ and let $\rho_i$ be its eigenray in the direction $d$. Let $Q'$ be the point at distance $100$ from $Q$ on $\rho_1$. By assumption on $\psi_i$ we have $[Q,Q']\subseteq \rho_i$ for all $i\ge 1$.

Since $T$ is a free $F_N$-tree with dense orbits, branch points are dense in every non-degenerate segment in $T$, particularly in $[Q,Q']$.
Let $P\in [Q,Q']$ be a branch point such that $0<d(Q,P)<10$.

By Corollary~\ref{cor:dense} such a point $P$ always exists since we
assumed that $\lambda(K_T)\ne \{1\}$ and thus branch points of $T$ are
dense in every nondegenerate segment of $T$.

Let $S\in [Q,Q']$ be such that $d(Q,S)=11$.

We have $H_i([P,S])\subseteq [Q,Q']$ for every $i\ge 1$, since the
stretching factors $\lambda(\phi_i)$ belong to the interval
$[2,2.01]$. Let $d_1$ be the direction at $P$ determined by $[P,Q']$.

Let $H_i'$ be the homotheties representing $\phi_i$ and fixing $P$ and
let $\rho_i'$ be their eigenrays in the direction $d_1$. Note that
$\rho_1'$ has an overlap of positive length $\ge c>0$ (for some
$0<c\le 1$) with $[P,Q']$. Since each $\rho_i'$ has overlap of length
$\ge 100$ with $\rho_1'$, it follows that each $\rho_i'$ has overlap
of length $\ge c$ with $[P,Q']$. Put $P_i=H_iP\in [P,Q']$, and let
$d^i_2$ be the direction at $P_i$ determined by $[P_i,Q']$.

Since $S \in d_1$ and $H_i([P,S])\subseteq [Q,Q']$, it follows that
$H_i(d_1) = d^i_2$.  Since $H_i$ preserves $F_N$-orbits of branch
points and of directions at branch points, it follows that for every
$i\ge 1$ there is $w_i\in F_N$ such that $P_i=w_iP$ and $d^i_2 = w_i
d_1$.

For every $i\ge 1$ let $H_i''$ be the homothety representing $\phi_i$
and fixing $P_i$. Let $\rho_i''$ be the eigenray of $H_i''$ in the
direction $d^i_2$.

Then part (1) of Proposition~\ref{well-known1} implies that
$\rho_i''=w_i\rho_i'$.  On the other hand, recall from
Corollary~\ref{cor:unique} that $H_i'$ has the
form $H_i'=v_iH_i$ for some $v_i\in F_N$.

Consider the ray $H_i(\rho_i')$. Part (2) of
Proposition~\ref{well-known1} implies that $H_i(\rho_i')$ is the
eigenray of the homothety $H_i H_i' H_i^{-1}$ that fixes $H_i(P)=P_i$
and in the direction $H_i(d_1) = d^i_2$.

Also,
\[
H_i H_i' H_i^{-1}= H_i v_i H_i H_i^{-1}=H_i v_i =\Phi_i(v_i)H_i
\]
so that this homothety represents $\phi_i$. Since furthermore it fixes $P_i$,
it must be equal to $H_i''$. It follows that $H_i(\rho_i')=\rho_i''=w_i\rho_i'$.

Note that by construction the ray $H_i(\rho_i')$ has overlap of length
$\ge 2c$ with $[P_i,Q']$. Let $Z\in [P,Q']$ be such that
$d(P,Z)=c$. Then we have $w_i[P,Z]\subseteq [P_i,Q]$ for every $i\ge
1$.

It follows that $w_j^{-1}w_i$, as $i,j\to \infty$ almost fixes $[P,Z]$
and acts on a fixed nondegenerate subsegment $J$ of the interior of
$[P,Z]$ with positive translation length which tends to zero. This
easily leads to a contradiction with the fact that the action of $F_N$
on $T$ is free. Indeed, the points $P_i=H_iP$ converge to the point
$P_\infty\in [Q,Q']$ where $d(Q,P_\infty)=2d(Q,P)$. Hence the
commutators $[w_j^{-1}w_i, w_k^{-1}w_l]$ fix $J$ pointwise as
$i,j,k,l\to\infty$ and therefore $[w_j^{-1}w_i, w_k^{-1}w_l]=1$. Since
$F_N$ is free, it follows that $w_j^{-1}w_i$ belong to a common
maximal cyclic subgroup $\langle g\rangle\le F_N$. However, this
contradicts the fact that $w_j^{-1}w_i$ can be made to have
arbitrarily small positive translation length.
\end{proof}

We can now prove the main result of this section. We define:
 \[
 P_T=\{\phi\in \Stab_{Out(F_N)}([T])
 \mid
\lambda(\phi)=1\}
\]

\begin{thm}\label{thm:vc}
Let
$T\in\cvnbar$ be a tree
with
free $F_N$-action.
Then we have:

\begin{enumerate}
\item[(a)]

The group $P_T$ is finite, and
the stabilizer $\Stab_{Out(F_N)}([T])$ is virtually cyclic. Moreover:
\[
\Stab_{Out(F_N)}([T])=\begin{cases} P_T \quad \quad \, \, \,  \, \, \text{ if }
  \Stab_{Out(F_N)}([T]) \text{ is finite}, \\
P_T\rtimes \mathbb Z \quad \text{ if }
  \Stab_{Out(F_N)}([T]) \text{ is infinite}\end{cases}
\]

\item[(b)]
If $T$ does not have dense $F_N$-orbits of branch points, then $\Stab_{Out(F_N)}([T])$ is finite.
\end{enumerate}
\end{thm}

\begin{proof}
(a)
Let $K_T \subseteq \Stab_{Out(F_N)}([T])$ be a normal subgroup of finite index
chosen as in Convention~\ref{conv:K_T}
and recall that $\lambda:\Stab_{Out(F_N)}([T])\to \mathbb R_{>0}$ is the stretching factor homomorphism.
Then we know from Proposition~\ref{finite} that $Ker(\lambda|_{K_T})=\{1\}$.
Moreover,
Proposition~\ref{discrete-image} implies that $\lambda(K_T)\subseteq \mathbb R_{>0}$ is cyclic. Therefore $\Stab_{Out(F_N)}([T])$ is
virtually cyclic.

Since by Corollary~\ref{cor:K} the subgroup $K_T$ has finite index in $\Stab_{Out(F_N)}([T])$, we see that $P_T$ contains
$Ker(\lambda|_{K_T})=\{1\}$ as a subgroup of finite index and thus
$P_T$ is finite.

Also, since every $\phi\in \Stab_{Out(F_N)}([T])$ with
$\lambda(\phi)\ne 1$ has infinite order, we conclude that
$P_T=\Stab_{Out(F_N)}([T])$ if and only if $\Stab_{Out(F_N)}([T])$ is
finite, that is, if and only if $\lambda\big(\Stab_{Out(F_N)}([T])\big)= \{1\}$.

Additionally, $P_T
%\trianglelefteq
\triangleleft \Stab_{Out(F_N)}([T])$ is normal in
$\Stab_{Out(F_N)}([T])$. 

Thus if
$\lambda\big(\Stab_{Out(F_N)}([T])\big)\ne \{1\}$, then
$\lambda\big(\Stab_{Out(F_N)}([T])\big)$ is infinite cyclic and for
any $\phi\in \Stab_{Out(F_N)}([T])$ mapped by $\lambda$ to the
generator of $\lambda\big(\Stab_{Out(F_N)}([T])\big)$ we have $P_T\cap
\langle \phi\rangle=\{1\}$. Therefore
$\Stab_{Out(F_N)}([T])=P_T\rtimes \langle \phi\rangle$ in this case.

\smallskip
\noindent
(b) Suppose now that $T$ is a minimal free $F_N$-tree which does not
have dense orbits of branch-points. We claim that $\lambda(K_T)=\{1\}$ in this
case. Indeed, suppose not. Then there exists $\phi\in K_T$ with
$0<\lambda(\phi)<1$. Let $\Phi\in \Aut(F_N)$ be a lift of $\phi$ to
$\Aut(F_N)$ and let $H$ be a $\lambda(\phi)$-homothety representing
$\Phi$.

Since $T$ does not have dense $F_N$-orbits of branch-points, the canonical simplicial metric quotient tree $T/(T^i_{dense})_i$, obtained from $T$ by contracting every maximal subtree $T^i_{dense}$ where an $\FN$-orbit is dense, is a non-trivial $\R$-tree with isometric $\FN$-action that has trivial stabilizers
(see \cite{Le} for more details). Hence it has only finitely many
orbits of edges.  On the other hand, the union of maximal open
nondegenerate segments in $T$, whose interiors do not contain any
branch point, is mapped by the $\FN$-equivariant quotient map $T \to
T/(T^i_{dense})_i$ injectively to the union of open edges in
$T/(T^i_{dense})_i$. Hence in $T$ there exist only finitely many $F_N$-orbits of maximal closed nondegenerate segments
whose interiors do not contain any branch points, and moreover, there is at least one such segment. In particular, the lengths of maximal closed nondegenerate segments in $T$, whose interiors do not contain any branch point, are bounded below by some constant $c>0$.
Choose a maximal nondegenerate segment $[a,b]\subseteq T$ whose interior does not contain any branch points of $T$. Since $H$ is a homothety of $T$, for every $n\ge 1$ the segment $H^n[a,b]$ has the same property: it is a maximal closed nondegenerate segments in $T$ whose interiors do not contain any branch points. However, the length of $H^n[a,b]$ is $\lambda^n(\phi)d(a,b)$ which converges to $0$ as $n\to\infty$, yielding a contradiction.

Thus indeed $\lambda(K_T)=\{1\}$. Since we already know that $Ker(\lambda|_{K_T})=\{1\}$, it follows that $K_T=\{1\}$.
Since $K_T$ has finite index in $\Stab_{Out(F_N)}([T])$, we conclude that $\Stab_{Out(F_N)}([T])$ is finite, as required. This completes the proof of the theorem.
\end{proof}

As pointed out in the Introduction, since by a result of~\cite{WZ} every finite subgroup of $\Out(F_N)$ has order at most $N!2^N$, it follows that in Theorem~\ref{thm:vc} we have $|P_T|\le N!2^N$, so that $\Stab_{Out(F_N)}([T])$ always has a cyclic subgroup (possibly trivial) of index at most $N!2^N$.

%\begin{rem}
%\label{roots}
%(a) In order to interpret  Theorem \ref{thm:vc} correctly we would like to point out that, in the case where $\Stab_{Out(F_N)}([T])$ is infinite, the canonical cyclic factor $\Z$ of $\Stab_{Out(F_N)}([T])$ contains the group $K_T \subseteq \Stab_{Out(F_N)}([T])$, but it may well be strictly larger: The
%{\em canonical}
%generator of this cyclic group is a (possibly not uniquely determined) automorphism $\phi_T \in \Stab_{Out(F_N)}([T])$ which has among all elements of $\Stab_{Out(F_N)}([T])$ the smallest positive stretching factor; but a priori there is no reason to expect that the corresponding homothety (or rather ``homotheties'') preserve the $\FN$-orbits of directions or even of branch points.

%\smallskip
%\noindent
%(b)
%On the other hand, one should also note that, although
%such a canonical
%the above
%generator $\phi_T$ can not have a proper root in $\Stab_{Out(F_N)}([T])$, it may well have one in $\Out(\FN)$: For example $\phi_T$ could be the ``double'' of an iwip automorphism $\alpha$ (compare Definition \ref{defn:iwips}) constructed by doubling the rank of the free group, while $T$ is topologically but not metrically the double of the limit tree $T_+(\alpha)$. Hence the exchange of the two factor trees homeomorphic to $T_+(\alpha)$ would not be an isometry of $T$, so that the automorphism, which acts as $\alpha$ on one factor and on the identity on the other and subsequently exchanges the factors, would be a square root of $\phi_T$ outside of $\Stab_{Out(F_N)}([T])$.
%\end{rem}

\begin{rem}
\label{non-free-actions}
(a) The statement of Proposition~\ref{discrete-image} holds also in the case where the $\FN$-action on $T$ is only very small and not necessarily free.  The proof follows the same lines, but gets technically a little more involved.

\smallskip
\noindent
(b) The conclusion of Theorem~\ref{thm:vc}, however,  becomes wrong if one omits the hypothesis that the $\FN$-action on $T$ is free. Easy counterexamples are provided for example by simplicial trees $T$ with trivial edge stabilizers and large vertex stabilizers: As those are automatically free factors, one has many automorphisms which act non-trivially on some of the vertex groups, but leave invariant the free product structure that is realized by $T$.
\end{rem}

\section{Tits alternative for dynamically large subgroups of
  $\Out(F_N)$}\label{sect:TA}

In this section we apply Theorem~\ref{thm:vc} to
give a new proof of the Tits
alternative for subgroups of $\Out(F_N)$ which contain an iwip automorphism.

\begin{defn}
\label{defn:iwips}
%$ $
%
\noindent (a)  An outer automorphism $\phi\in \Out(F_N)$ is called {\em irreducible with irreducible powers (iwip)} if no positive power of $\phi$ preserves the conjugacy class of a
proper free factor of $F_N$.

\smallskip
\noindent
(b)
An outer automorphism  $\phi\in \Out(F_N)$ is
called \emph{atoroidal} if it has no non-trivial periodic conjugacy classes, i.e. if there do not exist $t\ge 1$ and $w\in F_N-\{1\}$ such that
$\phi^t$ fixes the conjugacy class $[w]$ of $w$ in $F_N$.
\end{defn}

It was proved in~\cite{BH92} that for $N\ge 2$ an iwip automorphism $\phi \in \Out(\FN)$ is non-atoroidal if and only if $\phi$ is induced, via an identification of $\FN$ with the fundamental group of a compact  surface $S$ with a single boundary component, by a pseudo-Anosov homeomorphism $h: S \to S$.

\begin{rem}\label{defn:irr}
The terminology ``iwip'' derives from the groundbreaking paper \cite{BH92}: Bestvina-Handel call
an element $\phi\in \Out(F_N)$ is \emph{reducible} if there
exists a free product decomposition $F_N=C_1\ast\dots C_k\ast F'$,
where $k\ge 1$ and $C_i\ne \{1\}$, such that $\phi$ permutes the
conjugacy classes of subgroups $C_1,\dots, C_k$ in $F_N$. An element
$\phi\in \Out(F_N)$ is called \emph{irreducible} if it is not
reducible.

It is not hard to see that an element $\phi\in \Out(F_N)$ is an
iwip if and only if for every $n\ge 1$ the power $\phi^n$ is irreducible (sometimes such automorphisms are also called \emph{fully irreducible}).
\end{rem}

It is known by a result of Levitt and Lustig~\cite{LL} that
iwips have a simple ``North-South''
dynamics on the compactified Outer space $\CVNbar$:

\begin{prop}\label{prop:LL}\cite{LL}
Let $\phi\in \Out(F_N)$ be an iwip. Then there exist unique $[T_+]=[T_+(\phi)],[T_-]=[T_-(\phi)]\in \CVNbar$ with the following properties:
\begin{enumerate}
\item The elements $[T_+],[T_-]\in \CVNbar$ are the only fixed points of $\phi$ in $\CVNbar$.
\item For any $[T]\in \CVNbar$, $[T]\ne [T_-]$ we have $\lim_{n\to\infty} [T\phi^n]=[T_+]$ and for any $[T]\in \CVNbar$, $[T]\ne [T_+]$ we have $\lim_{n\to\infty} [T\phi^{-n}]=[T_-]$.

\item We have $T_+\phi=\lambda_+T$ and $T_-\phi^{-1}=\lambda_-T_-$   where $\lambda_+>1$ and $\lambda_->1$. Moreover $\lambda_+$ is the Perron-Frobenius eigenvalue of any train-track representative of $\phi$ and $\lambda_-$ is the Perron-Frobenius eigenvalue of any train-track representative of $\phi^{-1}$.
\end{enumerate}
\end{prop}

In~\cite{LL} it is also proved that convergence in (2) is locally
uniform and hence uniform on compact subsets.  Alternatively, we
proved in \cite{KL4}, using geodesic currents and the intersection
form, that pointwise North-South dynamics for the action of an
atoroidal iwip $\phi$ on $\CVNbar$ and on $\mathbb PCurr(F_N)$ already
implies that the convergence in part (2) of Proposition~\ref{prop:LL}
is uniform on compact subsets.

We give a precise statement:

\begin{prop}\label{prop:LL+}\cite{LL, KL4}
Let $\phi\in \Out(F_N)$ be an iwip, and let $[T_+]=[T_+(\phi)]$ be as in Proposition \ref{prop:LL}.

Then for any compact subset $K\subseteq \CVNbar-[T_-]$ and any neighborhood $U$ of $[T_+]$ there exists $M\ge 1$ such that for every $n\ge M$ we have $K\phi^n\subseteq U$.
\end{prop}

We can now show:

\begin{cor}\label{cor:free}
Let $\phi,\psi\in \Out(F_N)$ be atoroidal iwips such that $[T_\pm(\phi)], [T_\pm(\psi)]$ are four distinct points in $\CVNbar$. Then there exists $M\ge 1$ such that for any $m,n\ge M$ the subgroup $\langle \phi^m, \psi^n\rangle\le \Out(F_N)$ is free of rank two with free basis $\phi^m, \psi^n$.
\end{cor}

\begin{proof}
In $\CVNbar$ choose
disjoint open neighborhoods $U_+, U_-, V_+, V_-$ of $[T_+(\phi)]$,  $[T_-(\phi)]$,  $[T_+(\psi)]$, $[T_-(\psi)]$ respectively. By
Proposition~\ref{prop:LL+}
there exists $M\ge 1$ such that for every $n\ge M$ we have $(\CVNbar-U_-)\phi^n\subseteq U_+$, $(\CVNbar-V_-)\psi^n\subseteq V_+$, $(\CVNbar-U_+)\phi^{-n}\subseteq U_-$, and $(\CVNbar-V_+)\psi^{-n}\subseteq V_-$. Then the standard ping-pong argument implies that for every $m,n\ge M$ the subgroup $\langle \phi^m, \psi^n\rangle\le \Out(F_N)$ is free with free basis $\phi^m, \psi^n$.
\end{proof}

A subgroup of $\Out(\FN)$ will be called {\em dynamically large} if it contains an atoroidal iwip automorphism. Such subgroups have many nice properties, and their ``negative'', {\em dynamically small} subgroups (i.e. subgroups without atoroidal iwips) seem to be rather special, compare
\cite{Handel-Mosher-new}.

\begin{thm}[Tits alternative for dynamically large subgroups]\label{thm:TA}

Let $G\le Out(F_N)$ be a subgroup such that there exists an
atoroidal iwip $\phi\in G$. Let $[T_+(\phi)],[T_-(\phi)]\in \CVNbar$ be the attracting and repelling fixed points of $\phi$ in $\CVNbar$.
Then exactly one of the following occurs:
\begin{enumerate}
\item The group $G$ is virtually cyclic and preserves the set $\{[T_+(\phi)],[T_-(\phi)]\}\subseteq \CVNbar$.
\item The group $G$ contains an iwip $\psi=g\phi g^{-1}$ for
  some $g\in G$  such that
  $\{[T_+(\phi)],[T_-(\phi)]\}\cap\{[T_+(\psi)],[T_-(\psi)]\}=\emptyset$. Moreover, in this case there exists
an exponent
$M\ge 1$ such that the subgroup $\langle \phi^M, \psi^M\rangle\le G$ is free of rank two.
\end{enumerate}
\end{thm}

\begin{proof}
It is well-known (see, for example, \cite{GJLL}) that if $\phi$ is an atoroidal iwip, then $T_+(\phi)$ and  $T_-(\phi)$ are free $F_N$-trees.

Therefore by Theorem~\ref{thm:vc} we have
$Stab_{Out(F_N)}[T_+(\phi)]$ and
$Stab_{Out(F_N)}[T_-(\phi)]$ are virtually cyclic and contain
$\langle\phi\rangle$ as a subgroup of finite index.

If $G$ preserves the set  $\{[T_+(\phi)],[T_-(\phi)]\}$, then $G$ has
a subgroup of index at most 2 that fixes each of $[T_\pm (\phi)]$  and
hence $G$ is virtually cyclic. Thus we may assume that $G$ does not
preserve the set $\{[T_+(\phi)],[T_-(\phi)]\}$. So there exists $g\in G$ such
that $[T_{+}(\phi)]g\not \in \{[T_+(\phi)],[T_-(\phi)]\}$ or that
$[T_{-}(\phi)]g\not\in \{[T_+(\phi)],[T_-(\phi)]\}$. We assume the
former as the other case is symmetric. Thus $[T_{+}(\phi)]g\ne
[T_\pm(\phi)]$. Note that $\psi=g^{-1}\phi g\in G$ is also an
atoriodal iwip and that $[T_{+}(\psi)]= [T_{+}(\phi)]g$. We claim that
$[T_{-}(\phi)]\ne [T_\pm(\phi)]g$. Indeed, otherwise we have
$[T_-(g^{-1}\phi g)]=[T_+(\phi)]$ or $[T_-(g^{-1}\phi g)]=[T_-(\phi)]$ and hence $g^{-1}\phi g\in
Stab_{Out(F_N)}[T_+(\phi)]$ or $g^{-1}\phi g\in
Stab_{Out(F_N)}[T_-(\phi)]$. In either case (since both stabilizers
contain $\langle\phi\rangle$ as subgroup of finite index)
$g^{-1}\phi^kg=\phi^l$ for some $k\ne 0,l\ne 0$ and therefore
$g^{-1}\phi^k g$ has the same fixed points in $\CVNbar$ as does
$\phi^l$, namely, $[T_\pm (\phi)]$. This contradicts the fact that
$g^{-1}\phi g$ fixes the point $[T_+(\phi)]g\ne [T_\pm(\phi)]$. Thus
$[T_\pm(\phi)], [T_\pm(\psi)]$ are four distinct points. Therefore, by
Corollary~\ref{cor:free}  sufficiently high powers $\phi^M, \psi^M$ freely generate a free subgroup of rank two in $G$, as required.
\end{proof}

It is possible to prove Theorem \ref{thm:TA} also for subgroups $G\le Out(F_N)$
which contain an iwip $\phi\in G$ that is not atoroidal.  Indeed, precisely the same proof applies, except that in this case the limit trees $T_+(\phi)$ and $T_-(\phi)$ do not have a free $\FN$-action, so that we can not apply Theorem~\ref{thm:vc}
to show that for such $\phi$ the group $\Stab_{Out(F_N)}([T_+(\phi)])$ is virtually cyclic.  However, we can provide an alternative argument:

\begin{prop}\label{prop:toroidal}
Let $N\ge 2$ and let $\phi\in \Out(F_N)$ be an iwip which is not atoroidal. Then:
\begin{enumerate}
\item[(1)]
$\Stab_{Out(F_N)}([T_+(\phi)])$ is virtually cyclic.
\item[(2)] The same conclusion as in Theorem \ref{thm:TA} holds for subgroups of $\Out(F_N)$ which contain an
%arbitrary, and not necessarily atoroidal,
iwip $\phi$
which is not atoroidal.
\end{enumerate}
\end{prop}

\begin{proof}
As pointed out above, it suffices to prove statement (1). Part (2) then follows by exactly the same argument as in the proof of Theorem~\ref{thm:TA}. The only difference is that in the case where $G\le \Out(F_N)$ contains an iwip $\phi$ that is not atoroidal, in order to show that $\Stab_{Out(F_N)}([T_+(\phi)])$ is virtually cyclic we invoke part (1) of this proposition rather than Theorem~\ref{thm:vc}. Thus it is enough to establish (1):

  A result of Bestvina and Handel (see Theorem~4.1 in ~\cite{BH92})
  shows that if $\phi\in \Out(F_N)$ is an iwip which is not atoroidal
  then there exists a compact connected surface $S$ with a single
  boundary component and an identification $F_N=\pi_1(S, x_0)$ (where
  we assume that $x_0$ belongs to the boundary circle of $S$) such
  that $\phi$ is induced by a pseudo-Anosov homeomorphism $f$ of
  $S$. Let $c\in \pi_1(S)$ correspond to the boundary circle of
  $S$. Thus, if $S$ is orientable, then $N=2k$ is even and $F_N$
  has a free basis $a_1,b_1,\dots, a_k, b_k$ such that
  $c=[a_1,b_1]\dots [a_k,b_k]$. If $S$ is non-orientable, there is a
  free basis $a_1,\dots, a_N$ of $F_N$ such that $c=a_1^2a_2^2\dots
  a_N^2$.

  In this case $T_+:=T_+(\phi)$ is, up to rescaling, exactly the
  $\mathbb R$-tree $T_{\frak L}$ which is dual to the stable
  measured lamination $\frak L\in \mathcal{ML}(S)$ of the pseudo-Anosov $f$
  (see Ch. 11.12 in \cite{Kap} for details related to the construction
  of a dual $\mathbb R$-tree defined by a measured lamination on
  $S$). Moreover, by construction of $T_{\frak L}$, a nontrivial
  element $g\in F_N$ acts with a fixed point on $T_{\frak L}=T_+$
  if and only if $g$ is conjugate in $F_N$ to a nonzero power of $c$.

  Let $Mod(S)^\pm$ be the full mapping class group of $S$ (including
  isotopy classes of orientation-reversing homeomorphisms of $S$ if
  $S$ is orientable). It
  is well known that in this case $Mod(S)^\pm\le \Out(F_N)$ and in
  fact

\[
Mod(S)^\pm=\{\phi\in \Out(F_N): \phi([c])=[c^{\pm 1}]\}\le \Out(F_N).
\]

Let $\psi\in \Stab_{Out(F_N)}([T_+])$, that is $T_+\psi=\lambda T_+$
for some $\lambda>0$.

Since $||c||_{T_+}=0$, we have
\[
||\psi(c)||_{T_+}=||c||_{T_+\psi}=||c||_{\lambda T_+}=\lambda||c||_{T_+}=0.
\]
Thus $||\psi(c)||_{T_+}=0$ and hence $\psi(c)$ is conjugate to a power
of $c$ in $F_N$. Moreover, since $\psi$ is an (outer) automorphism, we
get $\psi([c])=[c]^{\pm 1}$. Hence $\psi\in Mod(S)^\pm\le
\Out(F_N)$. This shows that $\Stab_{Out(F_N)}([T_+])\le Mod(S)^\pm$
and in fact $\Stab_{Out(F_N)}([T_+])\le Stab_{Mod(S)^\pm}([T_{\frak
L}])$.

It is well known (see, for example, \cite{Be88}) that the map from the
space of projective measured laminations $\mathcal{PML}(S)$ to the
space of projectivized $\mathbb R$-trees, that takes an element of
$\mathcal{PML}(S)$ and sends it to the projective class of its dual
$\mathbb R$-tree is $Mod(S)^\pm$-equivariant and injective. Hence
$\Stab_{Mod(S)^\pm}([T_{\frak L}])=\Stab_{Mod(S)^\pm}([\frak
L])$. Since $\frak L$ is the stable measured lamination associated
to a pseudo-Anosov $f$, the group $\Stab_{Mod(S)^\pm}([\frak L])$ is
virtually cyclic (see Lemma 5.10 in \cite{Ivanov}). Therefore
$\Stab_{Out(F_N)}([T_+])$ is virtually cyclic, as claimed.
\end{proof}

Recall that if $G$ is a group and $H\le G$ is a subgroup, the \emph{commensurator} or \emph{virtual normalizer} of $H$ in $G$ is the subgroup
\[
Comm_G(H)=:\{g\in G\mid [H:H\cap g^{-1}Hg]<\infty, \text{ and } [g^{-1}Hg:H\cap g^{-1}Hg]<\infty.\}
\]
The commensurator $Comm_G(H)$ always contains the normalizer of $H$ in $G$.
As a consequence of Theorem~\ref{thm:vc} and part (1) of Proposition~\ref{prop:toroidal}, we obtain:

\begin{cor}\label{cor:comm}
Let $N\ge 2$ and let $\phi\in \Out(F_N)$ be an iwip. Then the commensurator $Comm_{Out(F_N)}(\langle \phi\rangle)$ is virtually cyclic.
\end{cor}

\begin{proof}
Let $g\in Comm_{Out(F_N)}(\langle \phi\rangle)$. For some non-zero integers $n,m$ we have $g^{-1}\phi^m g=\phi^n$.

Thus $\{[T_\pm(g^{-1}\phi^m g)]\}=\{[T_\pm(g^{-1}\phi g)]\}=\{[T_\pm(\phi)]g\}$ are the only two fixed points of the iwip $g^{-1}\phi^m g$ in $\CVNbar$. On the other hand $\{[T_\pm(\phi)]\}$ are the only two fixed points of the iwip $\phi^n$ in $\CVNbar$. Therefore $\{[T_\pm(\phi)]\}g=\{[T_\pm(\phi)]\}$, that is $g\in \Stab_{Out(F_N)}(\{[T_\pm(\phi)]\})$,
which gives
$Comm_{Out(F_N)}(\langle \phi\rangle)\le \Stab_{Out(F_N)}(\{[T_\pm(\phi)]\})$.

By Theorem~\ref{thm:vc} and part (1) of Proposition~\ref{prop:toroidal} the group $\Stab_{Out(F_N)}(\{[T_\pm(\phi)]\})$ is virtually cyclic. Hence $Comm_{Out(F_N)}(\langle \phi\rangle)$ is also virtually cyclic.
\end{proof}
Corollary~\ref{cor:comm} can also be derived, by a similar argument to the one given above, directly from Theorem~2.4 in \cite{BFH97}.

%%%%%%%%%%%%%%%%

%%%%%%%%%%%%%%%

\end{document}